\documentclass[oneside,reqno,british,english]{amsart}
\usepackage{mathpazo}
\usepackage[T1]{fontenc}
\usepackage[latin9]{inputenc}
\usepackage{color}
\usepackage{babel}
\usepackage{amstext}
\usepackage{amsthm}
\usepackage{amssymb}
\usepackage[unicode=true,pdfusetitle,
 bookmarks=true,bookmarksnumbered=false,bookmarksopen=false,
 breaklinks=false,pdfborder={0 0 0},pdfborderstyle={},backref=false,colorlinks=true]
 {hyperref}
\hypersetup{
 pdfborderstyle={},linkcolor=black,citecolor=black,urlcolor=blue}

\makeatletter
\theoremstyle{plain}
\newtheorem{thm}{\protect\theoremname}
\theoremstyle{remark}
\newtheorem*{claim*}{\protect\claimname}
\theoremstyle{plain}
\newtheorem{lem}[thm]{\protect\lemmaname}

\usepackage{calrsfs}
\usepackage{graphics}
\usepackage{color}
\usepackage{perpage}
\usepackage{romanbar}
\usepackage{breakurl}

\MakePerPage{footnote}

\gdef\SetFigFontNFSS#1#2#3#4#5{} %Silence pointless warnings due to xfig

\theoremstyle{remark}
\newtheorem*{qst*}{Question}
\theoremstyle{plain}

\ifpdf
\def\afs#1#2{\href{#1}{\nolinkurl{#2}}}
\else
\def\afs#1#2{\burlalt{#1}{#2}}
\fi
\let\@wraptoccontribs\wraptoccontribs
\makeatother

\makeatletter
\def\@setauthors{%
  \begingroup
  \def\thanks{\protect\thanks@warning}%
  \trivlist
  \centering\footnotesize \@topsep30\p@\relax
  \advance\@topsep by -\baselineskip
  \item\relax
  \author@andify\authors
  \def\\{\protect\linebreak}%
  \begin{center}
    \normalsize
    \textsc{Gady Kozma}$^*$ \textsc{and}
    \textsc{Alexander Lubotzky}$^{**}$
  \end{center}
  \bigskip \bigskip \bigskip
  \begin{minipage}[t]{0.45\textwidth}
    \normalsize
    $*$Department of Mathematics,\\
    The Weizmann Institute of Science,\\
    Rehovot 76100,\\
    Israel.\\
    \tt{gady.kozma@weizmann.ac.il}
  \end{minipage}
  \qquad
  \begin{minipage}[t]{0.45\textwidth}
    \normalsize
    $^{**}$Einstein Institute of Mathematics,\\
    The Hebrew University,\\
    Jerusalem 91904,\\
    Israel.\\
    \tt{alex.lubotzky@mail.huji.ac.il}
  \end{minipage}
  \ifx\@empty\contribs
  \else
    ,\penalty-3 \space \@setcontribs
    \@closetoccontribs
  \fi
  \endtrivlist
  \endgroup
}
\makeatother

\makeatother

\addto\captionsbritish{\renewcommand{\claimname}{Claim}}
\addto\captionsbritish{\renewcommand{\lemmaname}{Lemma}}
\addto\captionsbritish{\renewcommand{\theoremname}{Theorem}}
\addto\captionsenglish{\renewcommand{\claimname}{Claim}}
\addto\captionsenglish{\renewcommand{\lemmaname}{Lemma}}
\addto\captionsenglish{\renewcommand{\theoremname}{Theorem}}
\providecommand{\claimname}{Claim}
\providecommand{\lemmaname}{Lemma}
\providecommand{\theoremname}{Theorem}

\begin{document}
\title{Linear representations of random groups}
\author{Gady Kozma and Alexander Lubotzky}
\begin{abstract}
We show that for a fixed $k\in\mathbb{N}$, Gromov random groups with
any density $d>0$ have no non-trivial degree $k$ representations
over any field, a.a.s. This is especially interesting in light of
the results of Agol, Ollivier and Wise that when $d<\frac{1}{6}$
such groups have a faithful linear representation over $\mathbb{Q}$,
a.a.s.
\end{abstract}

\maketitle

\section{Introduction}

\subsection*{Gromov random groups.}

Let $m\ge2$ and let $F$ be the free group on $m$ generators $x=\{x_{1},\dotsc,x_{m}\}$.
For $l\in\mathbb{N}$, let $S_{l}$ be the sphere of radius $l$ in
the Cayley graph of $F$ with respect to $X$, i.e.\ the set of reduced
words in $x_{i}^{\pm1}$ of length $l$. Fix some $d\ge0$ and let
$R$ be a random subset of $S_{l}$ constructed by taking $\lfloor|S_{l}|^{d}\rfloor=\lfloor\big(2m(2m-1)^{l-1}\big)^{d}\rfloor$
elements of $S_{l}$ uniformly, independently and with repetitions.
The group $\Gamma=\langle x|R\rangle$ i.e.\ the group presented
by the generators $x$ and the relators $R$ is called a ``Gromov
random group of density $d$ with $m$ generators and relators of
length $l$''. For a group property $P$, we say that Gromov random
groups satisfy $P$ asymptotically almost surely (a.a.s.) if the probability
of $P$ goes to $1$ as $l\to\infty$. In a formula,
\[
\lim_{l\to\infty}\mathbb{P}(\Gamma\text{ satisfies }P)=1.
\]
See \cite{O05} for an invitation to the topic. The goal of this note
is to prove
\begin{thm}
\label{thm:d}Let $k\ge1$, $m\ge2$ and $d>0$. Then Gromov random
groups $\Gamma$ at density $d$ with $m$ generators satisfy a.a.s.\ that
for any field $F$ and any $\rho:\Gamma\to\mathrm{GL}_{k}(F)$, $|\rho(\Gamma)|\le2$.
\end{thm}

In fact, we prove that polynomially many relators are enough for this
property, see the formulation of Theorem \ref{thm:poly} below.

When $l$ is odd, it is easy to see that $\mathbb{Z}/2\mathbb{Z}$
is a.a.s.\ not a quotient of $\Gamma$ hence in fact we may strengthen
Theorem \ref{thm:d} to state that $\rho(\Gamma)=\{1\}$. Similarly,
when $l$ is even $\mathbb{Z}/2\mathbb{Z}$ is (deterministically)
a quotient of $\Gamma$ so the possibility of an image of size 2 cannot
be removed.

We recall the well-known result of Gromov \cite[\S V]{O05} that for
a fixed $m\ge2$, $\Gamma$ is a.a.s.~an infinite hyperbolic group
for $d<\frac{1}{2}$, while $|\Gamma|\le2$ for $d>\frac{1}{2}$.
So our theorem is of interest only for $d\le\frac{1}{2}$. But it
is especially interesting for $d<\frac{1}{6}$. In this case Agol
\cite{A13} and Ollivier and Wise \cite{OW13} proved the following
remarkable result:
\begin{thm}
\label{thm:agol}For a fixed $m\ge2$ and $d<\frac{1}{6}$, the random
group $\Gamma$ is a.a.s.\ linear over $\mathbb{Z}$ i.e.\ has a
faithful representation into $\mathrm{GL}_{k}(\mathbb{Z})$, for some
$k\in\mathbb{N}$.
\end{thm}

Thus the main difference between these two results is whether $k$,
the degree of linearity, is allowed to depend on $l$, the length
of the relators, or not. If it is allowed, we are in the case of Theorem
\ref{thm:agol} and a representation exists. If it is fixed, we are
in the case of Theorem \ref{thm:d} and no representation exists.

A remark on the field: while Theorem \ref{thm:agol} constructs a
representation into $\mathbb{Q}$, in fact it implies arbitrarily
large representations for any field. We cannot show this by taking
the representation into $\textrm{GL}_{k}(\mathbb{Z})$ modulo $p$
as that might be trivial. But we can, instead, use the fact that any
subgroup of $\textrm{GL}_{k}(\mathbb{Z})$ is residually finite (simply
because $\cap_{m}\ker(\textrm{GL}_{k}(\mathbb{Z})\to\textrm{GL}_{k}(\mathbb{Z}/m\mathbb{Z}))=\{1\}$)
so has arbitrarily large finite quotients. These finite quotients
may be embedded into a symmetric group, hence for some $k'$ it will
embed (as permutation matrices) into $\mathrm{GL}_{k'}(F)$ for any
$F$. 

\section{\label{sec:AG}Algebraic geometry preliminaries}

The proof uses some results from algebraic geometry. We will now survey
briefly the notions and results we need, assuming only that the reader
is familiar with undergraduate algebra. 

Let $F$ be an algebraically closed field of any characteristic, and
let $n\ge0$. A subset $W$ of the affine space $\mathbb{A}^{n}:=F^{n}$
is called an \emph{(affine) variety} if 
\[
W=\bigcap_{i=1}^{k}Z(p_{i})\qquad Z(p)=\{x\in\mathbb{A}^{n}:p(x)=0\},
\]
where $p_{1},\dotsc,p_{k}$ are polynomials in $n$ variables. We
will use the notations $F$ (for the underlying algebraically closed
field) and $Z(p)$ throughout the paper.

A variety is called \emph{irreducible} if it cannot be written as
a union of two proper subvarieties. Any variety can be written as
a finite union of irreducible varieties. Assuming that the representation
is not redundant (i.e.\ if $W=\bigcup X_{i}$ then $X_{i}\nsubseteq X_{j}$
for any $i\ne j$), it is unique. The $X_{i}$ of this unique representation
are called the \emph{irreducible components} of $W$. See \cite[theorems 1.4 \& 1.5]{Sh}.
Let us remark that in some of the literature, including \cite{H77,Sh},
a variety is defined to be automatically irreducible. But for us it
will be convenient to define it as above.

For any affine variety one can define its \emph{dimension}, denoted
by $\dim$.  Heuristically it corresponds with the natural notion
of dimension, but the formal definition requires some preliminaries
which we prefer to skip. The reader may consult \cite[chapter 1, \S 6]{Sh}.
We will need the following properties of it:
\begin{enumerate}
\item $\dim(W)\in\{-1,0,1,2,\dotsc\}.$
\item $\dim(\mathbb{A}^{n})=n$.
\item $\dim(W)=-1$ only for $W=\emptyset$ and $\dim(W)=0$ implies that
$W$ is finite.
\item \label{enu:dim decrease}If $W$ is an irreducible algebraic variety
with $\dim(W)=k$ and if $p$ is a polynomial not identically zero
on $W$, then any irreducible component of $W\cap Z(p)$ has dimension
$k-1$. 
\end{enumerate}
See \cite[Corollary 1.13]{Sh} for this last, and most remarkable
property. (Note that theorem numbering in the third edition of Shafarevich
is different from those of the previous editions).

The following result will be referred to as ``B\'ezout's theorem''
(the literature is abound with results called ``B\'ezout's theorem'',
some of them very close in formulation to it, so we are certainly
following tradition here). 
\begin{thm}[B\'ezout]
\label{thm:Bezout}Let $W$ be an affine variety defined by polynomials
$f_{1},f_{2},\dotsc,f_{m}$ in $n$ variables i.e.\ $W=\cap Z(f_{i})\subset\mathbb{A}^{n}$.
Suppose $\deg f_{i}\le d$ for all $i$. Then the number of irreducible
components of $W$ is bounded by $d^{\min(m,n)}$.
\end{thm}

This result is well-known, even classic. And yet we could not find
a reference to it in this form. Hence we supply a proof.
\begin{proof}
The literature is far more complete for \emph{projective} varieties.
Hence our first step will be to define the projective space $\mathbb{P}^{n}$
and show how the projective B\'ezout theorem implies the affine one
(we hope no confusion will arise from the use of $\mathbb{P}$ for
``probability'' in other parts of the paper. The use of $\mathbb{P}$
for the projective space will be restricted to the proof of Theorem
\ref{thm:Bezout}).

The projective space $\mathbb{P}^{n}$ over a field $F$ is the space
$F^{n+1}\setminus\{0\}$ (we consider the coordinates $0,\dotsc,n$),
modulo the relation $v\sim av$ for every $v\in F^{n+1}\setminus\{0\}$
and $a\in F\setminus\{0\}$. A projective variety is the intersection
of \foreignlanguage{british}{zeroes} of \emph{homogeneous} polynomials.
Irreducible projective varieties are defined like affine ones, and
the decomposition result that allows to define irreducible components
is as in the affine case (\cite[page 46]{Sh} claims that ``the proof
carries over word-for-word''). We will need two maps between subvarieties
of $\mathbb{A}^{n}$ and $\mathbb{P}^{n}$. The first, restriction,
takes the projective variety $\cap Z(f_{i})$, $f_{i}$ homogeneous
polynomials in $x_{0},\dotsc,x_{n}$ and maps it to the affine variety
$\cap Z(g_{i})$ where $g_{i}(x_{1},\dotsc,x_{n})=f_{i}(1,x_{1},\dotsc,x_{n})$.
The second, \foreignlanguage{british}{homogenisation}, maps an affine
variety $\cap Z(g_{i})$ into the projective variety $\cap Z(f_{i})$
where $f_{i}$ are produced from $g_{i}$ by taking every monomial
$ax_{1}^{b_{1}}\dotsb x_{n}^{b_{n}}$ of $g_{i}$ and mapping it to
$ax_{0}^{b_{0}}\dotsb x_{n}^{b_{n}}$ where $b_{0}=\deg g_{i}-(b_{1}+\dotsb+b_{n})$,
and summing those to get $f_{i}$. We will denote ``$W$ is the restriction
of $V$'' by $W=V\cap\mathbb{A}^{n}$, and ``$V$ is the homogenisation
of $W$'' by $V=\overline{W}$. Clearly $\overline{W}\cap\mathbb{A}^{n}=W$
for any affine $W$.
\begin{claim*}
The restriction of an irreducible projective variety is irreducible.
\end{claim*}
\begin{proof}
Let $V$ be the irreducible projective variety, and let $W=V\cap\mathbb{A}^{n}$.
Assume by contradiction that $W=W_{1}\cup W_{2}$ in a non-trivial
way. We now claim that $(\overline{W_{1}}\cap V)\cup(\overline{W_{2}}\cap V)\cup(\{x_{0}=0\}\cap V)$
is a non-trivial decomposition of $V$. Indeed, this is clearly a
decomposition of $V$, and it is non-trivial because any $(x_{1},\dotsc,x_{n})\in W_{1}\setminus W_{2}$
would satisfy that $(1,x_{1},\dotsc,x_{n})\in\overline{W_{1}}\setminus\overline{W_{2}}$,
and similarly for $W_{2}\setminus W_{1}$. 
\end{proof}
With the claim, the affine B\'ezout theorem follows from the projective
one as follows: Let $W=\cap_{i=1}^{m}Z(f_{i})$ with $\deg f_{i}\le d$.
Then $\overline{W}$ has the same structure, and hence by the projective
B\'ezout theorem its decomposition to irreducible components $\overline{W}=X_{1}\cup\dotsb\cup X_{K}$
satisfies $K\le d^{\min(n,m)}$. By the claim, $X_{i}\cap\mathbb{A}^{n}$
are irreducible, and of course 
\[
W=\overline{W}\cap\mathbb{A}^{n}=(X_{1}\cup\dotsb\cup X_{K})\cap\mathbb{A}^{n}=(X_{1}\cap\mathbb{A}^{n})\cup\dotsb\cup(X_{K}\cap\mathbb{A}^{n})
\]
which is a decomposition of $W$ to irreducible subvarieties (it might
be redundant, but that would only means the number of components of
$W$ is smaller than $K$). Thus we need only show the projective
B\'ezout theorem.

For the projective B\'ezout theorem we will need the concepts of
the dimension and degree of a projective variety. The dimension of
a projective variety is as for an affine variety, and has the same
four properties listed above ($\dim(\mathbb{P}^{n})=n$), with the
same references in \cite{Sh}. As for the degree, heuristically if
$W\subset\mathbb{P}^{n}$ is some irreducible variety then $\deg W$
is the number of intersections of $W$ with a generic linear variety
of dimension $n-\dim W$. Again, the formal definition is different
and we will skip it, the reader may consult \cite[page 50]{H77}.
We only need the following properties to use Theorem \ref{thm:essentially Bezout}
below:
\begin{enumerate}
\item $\deg(W)$ is always a positive integer, except $\deg(\emptyset)=0$. 
\item $\deg(\mathbb{P}^{n})=1$.
\end{enumerate}
For both properties, see \cite[Chapter 1, Propsition 7.6]{H77}. The
projective B\'ezout theorem follows as a corollary from the following
result:
\begin{thm}
\label{thm:essentially Bezout}Let $W$ be an irreducible projective
variety. Let $f$ be a homogeneous polynomial. Let $X_{1},\dotsc,X_{s}$
be the irreducible components of $W\cap Z(f)$. Then
\[
\sum_{j=1}^{s}\deg(X_{j})\le\deg W\cdot\deg f
\]
\end{thm}

Where $\deg W$ is the degree of a projective variety just mentioned,
while $\deg f$ is the usual degree of a polynomial. See \cite[Theorem 7.7 and Proposition 7.6d]{H77}.
The formulation in \cite{H77} has some additional quantities, intersection
multiplicities, denoted by $i(\cdot)$ \textemdash{} all we need from
them is that they are at least $1$, which follows because they are
defined as lengths of some modules (\cite{H77}, top of page 53 and
the definition at page 51), and the length of a module is the maximal
size of a decreasing sequence of submodules. The formulation in \cite{H77}
requires that $\dim W\ge1$ and that $f$ is not identically zero
on $W$, but the case $\dim W=0$ (i.e.\ $W$ is a single point)
is obvious, and so is the case $W\subset Z(f)$.

Let now $f_{1},\dotsc,f_{m}$ be polynomials with $\deg f_{i}\le d$,
and let 
\begin{equation}
\bigcap_{i=1}^{m}Z(f_{i})=X_{1}\cup\dotsb\cup X_{K}\label{eq:Xi}
\end{equation}
be the decomposition of $\cap Z(f_{i})$ into irreducible components.
We claim that
\begin{equation}
\sum_{j=1}^{K}\deg(X_{j})d^{\dim X_{j}}\le d^{n}.\label{eq:Bezout}
\end{equation}
We show (\ref{eq:Bezout}) by induction on $m$. Indeed, $m=0$ is
obvious. Assume (\ref{eq:Bezout}) has been proved for $m$ and write
(using the $X_{i}$ of (\ref{eq:Xi}))
\[
\bigcap_{i=1}^{m+1}Z(f_{i})=\bigcup_{j=1}^{K}X_{j}\cap Z(f_{m+1}).
\]
Fix $j$ and let $Y_{j,k}$ be the irreducible components of $X_{j}\cap Z(f_{m+1})$.
By Theorem \ref{thm:essentially Bezout},
\[
\sum_{k}\deg(Y_{j,k})\le d\deg(X_{j}).
\]
If $X_{j}\nsubseteq Z(f_{m+1})$ then by property 4 of the dimension,
$\dim Y_{j,k}=\dim X_{j}-1$ so 
\begin{equation}
\sum_{k}\deg(Y_{j,k})d^{\dim Y_{j,k}}\le\deg(X_{j})d^{\dim X_{j}}.\label{eq:YandX}
\end{equation}
But if $X_{j}\subseteq Z(f_{m+1})$ then (\ref{eq:YandX}) holds trivially
(with no need to invoke Theorem \ref{thm:essentially Bezout}). So
(\ref{eq:YandX}) holds always. We sum (\ref{eq:YandX}) over $j$
to get
\[
\sum_{j,k}\deg(Y_{j,k})d^{\dim Y_{j,k}}\le\sum_{j}\deg(X_{j})d^{\dim X_{j}}\le d^{n}
\]
where the second inequality is the induction assumption. Now, $Y_{j,k}$
is a decomposition of $\cap_{i=1}^{m+1}Z(f_{i})$ to irreducible components
\textemdash{} it may be redundant, but that only reduces the sum in
(\ref{eq:Bezout}) further. Hence (\ref{eq:Bezout}) holds for $m+1$
and the induction is complete.

Theorem \ref{thm:Bezout} now follows easily. We drop the degrees
(as we may, as they are always at least $1$) and get
\[
\sum_{j=1}^{K}d^{\dim X_{j}}\le d^{n}
\]
If $m\ge n$ Theorem \ref{thm:Bezout} follows immediately. If $m<n$
it follows because then each $X_{j}$ has dimension at least $n-m$.
\end{proof}
The next result we need is an effective version of the nullstellensatz.
Hilbert's nullstellensatz is the following: Suppose $p_{i}$ are polynomials
in $n$ variables with $\cap Z(p_{i})=\emptyset$. Then there exists
polynomials $q_{i}$ such that $\sum p_{i}q_{i}\equiv1$.  There
is also a version of the nullstellensatz when $W:=\cap Z(p_{i})\ne\emptyset$.
It states that if $r$ is a polynomial which is zero on every point
of $W$, then there exists a $\nu\ge1$ and $q_{i}$ such that $\sum p_{i}q_{i}=r^{\nu}$.
These theorems hold for any field, but we will need them only for
$\mathbb{Q}$. Multiplying by the common denominator we get a result
that holds in $\mathbb{Z}$, i.e.\ if the $p_{i}$ and the $r$ have
integer coefficients then one may find polynomials $q_{i}$ with integer
coefficients, and integers $\nu$ and $b$ such that $\sum p_{i}q_{i}=br^{\nu}$.
We will need an effective version of this result but, in fact, the
only quantity we need to control is $b$. Hence the effective version
is as follows:
\begin{thm}
\label{thm:nullstellensatz}Let $p_{1},\dotsc,p_{t},r\in\mathbb{Z}[x_{1},\dotsc,x_{n}]$,
assume $r$ vanishes on $\cap_{i=1}^{t}Z(p_{i})$. Assume also that
$\deg p_{i}\le d$ $\forall i$, $\deg r\le d$ and all coefficients
of all $p_{i}$ are bounded by $h$. Then there exists $q_{i}\in\mathbb{Z}[x_{1},\dotsc,x_{n}]$,
$i=1,\dotsc,t$ and $b,\nu\in\mathbb{N}$ such that
\[
\sum_{i=1}^{t}p_{i}q_{i}=br^{\nu}
\]
with the bound
\[
\log b\le C^{n}n^{2n}(d+1)^{n(n+2)}(\log h+Cn^{2}\log d).
\]
\end{thm}

Here and below $C$ and $c$ will stand for absolute constants whose
value might change from line to line. We will only use the following,
rough bound, which holds for a fixed $n$ and $d$ sufficiently large
(i.e.\ $d>d_{0}(n)$),
\begin{equation}
\log b\le(2d)^{n(n+2)+1}\log h.\label{eq:ENsimp}
\end{equation}

\begin{proof}
We will find $q_{i}\in\mathbb{Q}[x_{1},\dotsc,x_{n}]$ such that $\sum p_{i}q_{i}=r^{\nu}$
and then $b$ will be bounded by the lcm of the denominators of the
$q_{i}$. By the corollary to Theorem 1 of \cite{B87}, we may take
$q_{i}\in\mathbb{Q}[x_{1},\dotsc,x_{n}]$ with
\begin{equation}
\deg q_{i}\le(n+1)(n+2)(d+1)^{n+2}=:Q.\label{eq:defQ}
\end{equation}
Once the degree is bounded, the coefficients of the $q_{i}$ are
given by the solution of a system of linear equations (depending on
the $p_{i}$, on $r$ and on $\nu$). Let $f(d,n)$ be the dimension
of the space of polynomials with $n$ variables and degree $\le d$
(so $f(d,n)\le(d+1)^{n}$). The system might be underdetermined, we
have $tf(Q,n)$ variables and at most $f(Q+d,n)$ equations, one for
each coefficient of one monomial in the equality $\sum p_{i}q_{i}=r^{\nu}$,
up to the degree of the left-hand side. Let $R$ be the rank of this
system of equations, so $R\le f(Q+d,n)$. Pick arbitrarily $R$ variables
and $R$ equations such that the corresponding submatrix $M$ is invertible
and solve the restricted equations. Set the rest of the variables
to zero, and the remaining equations (if any) will be fulfilled automatically.
It follows that some choice of the $q_{i}$ can be achieved by inverting
$M$ and applying the result to the vector of the coefficients of
$r^{\nu}$, themselves integers. Since $M^{-1}=M'/\det M$, where
$M'$ has integer entries, we may bound $b\le\det M$. But the entries
of $M$ are simply the coefficients of the $p_{i}$, all of them bounded
by $h$. Applying Hadamard's inequality (the determinant is bounded
by the product of the $l^{2}$ norms of the rows) gives
\[
b\le\left(h\sqrt{R}\right)^{R}
\]
or
\begin{align*}
\log b & \le R(\log h+\frac{1}{2}\log R)\le f(Q+d,n)(\log h+\frac{1}{2}\log f(Q+d,n))\\
 & \le(Q+d+1)^{n}(\log h+\frac{1}{2}n\log(Q+d+1))\\
 & \stackrel{\textrm{(\ref{eq:defQ})}}{\le}\left((n+1)(n+2)(d+1)^{n+2}+d+1\right)^{n}\cdot\\
 & \qquad\qquad\cdot\;(\log h+\frac{1}{2}n\log((n+1)(n+2)(d+1)^{n+2}+d+1))\\
 & \le C^{n}n^{2n}(d+1)^{n(n+2)}(\log h+Cn^{2}\log d)
\end{align*}
as claimed.
\end{proof}

\section{Proof of the main result}
\begin{lem}
\label{lem:half}Let $G$ be any $d$-regular connected multigraph
with $d\ge4$ and more than 2 vertices, and let $x$ be some vertex
of $G$. Let $t>1$. Then 
\[
\mathbb{P}^{x}(N(t)=x)\le\frac{d-2}{d-1}
\]
where $N(t)$ is a nonbacktracking random walk on $G$ at the $t^{\textrm{th}}$
step and $\mathbb{P}^{x}$ denotes the probability when $N(0)=x$.
\end{lem}

Let us define precisely what we mean by ``multigraph'' and ``nonbacktracking
random walk''. A multigraph is a graph which might contain multiple
edges and self-loops. It is $d$-regular if every vertex has exactly
$d$ edges connected to it, with a self-loop counted as two edges.
A nonbacktracking random walk is a walk that is not allowed to traverse
an edge and on the next step traverse it in the opposite direction
(there are no restrictions on the first step). A self-loop can be
traversed in either direction, and the nonbacktracking condition is
that it cannot be traversed and then traversed backwards. When the
multigraph is $d$-regular, this process has exactly $d-1$ possibilities
at each step (except the first one), and it chooses each with probability
$1/(d-1)$, independently of the past.
\begin{proof}
Fix the vertex $x$ for the rest of the proof. Every edge of our multigraph
we consider as two directed edges (a self-loop too corresponds to
two directed edges), and for a directed edge $e$ we denote by $\overline{e}$
the inverted edge. Hence, the non-backtracking condition is that the
walk is not allowed to traverse $e$ immediately after traversing
$\overline{e}$. (note that we have a multigraph, so there can be
$e\ne f$ that both go from vertex $x$ to vertex $y$. Still, we
may traverse $\overline{e}$ and then $f$, or $\overline{f}$ and
then $e$. It is only the couples $\overline{e},$$e$ and $\overline{f}$,
$f$ that are prohibited. Each self-loop corresponds to two directed
edges which are $\overline{\cdot\vphantom{e}}$ of one another). Let
$q_{t}(e)$ be the probability that the edge $e$ was traversed at
time $t$ i.e.\ if $e:v\to w$ (i.e., $e$ is from $v$ to $w$,
we will also use the notation $e:\to v$ and $e:v\to$ if we do not
care about the other vertex) then it is the probability that $N(t-1)=v$
and then the process continues through $e$ (which means, in particular,
that $N(t)=w$). Let $Q(t)=\max_{e}q_{t}(e)$. Then $Q(t)$ is non-increasing
because
\[
q_{t+1}(e)=\frac{1}{d-1}\sum_{f:\to v,\,f\ne\overline{e}}q_{t}(f)\le Q(t)\qquad\forall e:v\to.
\]
Examine now the event that $e:y\to z$ was traversed in the second
step. It requires that $N(1)=y$. Assume first (call this ``case
I'') that each neighbour $y$ of $x$ is connected to $x$ by $\le d-2$
edges (including $x$ itself, if there are self-loops). Then $\mathbb{P}(N(1)=y)\le(d-2)/d$
for every $y$ and hence $Q(2)\le(d-2)/(d(d-1))$. 

If $x$ has a neighbour $y$ to which it is connected by more than
$d-2$ edges (``case II''), then the requirements of regularity
and more than 2 vertices say that it must be connected to $y$ by
exactly $d-1$ edges, and further that it has a second neighbour to
which it is connected by $1$ edge. This means that $\mathbb{P}(N(2)=x)=(d-2)/d$
and, in particular, any vertex $z\ne x$ we have $\mathbb{P}(N(2)=z)\le2/d$.
Hence
\[
Q(3)\le\frac{\max\{2,(d-2)\}}{d(d-1)}
\]
Since $d\ge4$ we get the same bound as in case I. Hence $Q(t)\le(d-2)/(d(d-1))$
for all $t\ge3$. But this means that
\[
\mathbb{P}(N(t)=x)=\sum_{e:\to x}q_{t}(e)\le dQ(t)\le\frac{d-2}{d-1}.
\]
This covers all cases of the lemma except $t=2$ in case II, but we
just calculated that in this case $\mathbb{P}(N(2)=x)=(d-2)/d$. The
lemma is thus proved.
\end{proof}
The next lemma is quite close to the formulation of Theorem \ref{thm:poly},
the only difference is that it handles only one field.
\begin{lem}
\label{lem:one F}Let $F$ be an algebraically closed field, let $k,m\in\mathbb{N}$,
$m\ge2$ and let $l$ be sufficiently large (i.e.\ $l>l_{0}(k,m)$).
Let $R$ be given by taking $u\ge15m^{3}k^{4}\log l$ random reduced
words of length $l$ in the letters $\{x_{1},\dotsc,x_{m},x_{1}^{-1},\dotsc,x_{m}^{-1}\}$
independently, uniformly, with repetitions. Let $\Gamma=\langle x_{1},\dotsc,x_{m}|R\rangle$.
Then
\[
\mathbb{P}(\exists\rho:\Gamma\to\mathrm{GL}_{k}(F)\text{ such that }|\rho(\Gamma)|>2)\le\exp(-cu/mk^{2})
\]
where $\rho$ is a group homomorphism.
\end{lem}

(log here is the natural logarithm).
\begin{proof}
We consider $\mathrm{GL}_{k}(F)$ as a subvariety of $F^{2k^{2}}$
by considering the first set of $k^{2}$ variables as the entries
of the matrix and the second set of $k^{2}$ variables as the entries
of the inverse matrix, and adding polynomial equations ($k^{2}$ of
them, all of degree 2) that ensure that indeed, the product of the
two matrices is $1$. Similarly we consider $\textrm{GL}_{k}(F)\times\dotsb\times\textrm{GL}_{k}(F)$
($m$ times) as a subvariety of $F^{2mk^{2}}$. Denote this variety
by $X$. For $A\in\mathrm{GL}_{k}(F)\times\dotsb\times\mathrm{GL}_{k}(F)$
we denote $(A,A^{-1}):=(A_{1},A_{1}^{-1},\dotsc,A_{m},A_{m}^{-1})\in X$.

Let $E_{j}\subset X$ be the collection of $(A,A^{-1})$ such that
the matrices $A_{1},\dotsc,A_{m}$ satisfy the first $j$ words in
$R$. Since these (random) words can be thought of as (random) polynomial
equations in $2mk^{2}$ variables as above, $E_{j}$ is a (random)
variety in $F^{2mk^{2}}$. Let $(A,A^{-1})$ be a point in $E_{j}$
with $|\langle A\rangle|>2$. Examine the event that $(A,A^{-1})\in E_{j+1}$,
conditioned on $E_{j}$. The new reduced word $\omega$ that was added
to form $E_{j+1}$ is independent of the past, and hence $\omega(A)$
is distributed like a nonbacktracking random walk of length $l$ on
the Cayley graph generated by $A$ (if for some $i$ $A_{i}=1$ the
graph will contain one corresponding self-loop on each vertex. The
two directions of this self-loop will correspond to multiplying by
$A_{i}$ and $A_{i}^{-1}$. This matches with the definitions we gave
around Lemma \ref{lem:half}). This Cayley graph is a $2m$-regular
multigraph, and by assumption it has more than 2 vertices. Hence we
may use Lemma \ref{lem:half} to get
\[
\mathbb{P}((A,A^{-1})\in E_{j+1}\,|\,(A,A^{-1})\in E_{j})\le\frac{2m-2}{2m-1}.
\]
In other words, with probability $\ge\frac{1}{2m-1}$, adding one
relation breaks the irreducible component containing $(A,A^{-1})$
into further irreducible components, which then must have smaller
dimension, by property (\ref{enu:dim decrease}) of the dimension
(see \S \ref{sec:AG}).

Repeating this $\lambda$ times we get that after adding $\lambda$
words we break any fixed irreducible component with probability at
least $1-\exp(-\lambda/2m)$. By Bezout's theorem (Theorem \ref{thm:Bezout}),
$E_{j}$ has no more than $l^{2mk^{2}}$ irreducible components (the
initial polynomial equations defining $X$ have degree 2). Hence a
simple union bound shows that, for $\lambda\ge5m^{2}k^{2}\log l$,
\[
\mathbb{P}(\textrm{all components are broken})\ge1-l^{2mk^{2}}\exp\left(-\frac{\lambda}{2m}\right)\ge1-\exp\left(-\frac{\lambda}{10m}\right).
\]
Use that for $\lambda=u/(2mk^{2}+1)\ge5m^{2}k^{2}\log l$ words, and
get that with probability at least $1-\exp(-cu/mk^{2})$ one breaks
all components. Therefore the maximal degree decreases by $1$. Repeating
this a further $2mk^{2}+1$ times, the maximal degree of any component
which contains any $(A,A^{-1})$ with $|\langle A\rangle|>2$ is $-1$,
so they are in fact empty. We get the claim of the lemma with probability
$(2mk^{2}+1)\exp(-cu/mk^{2})$ but of course the outer $2mk^{2}+1$
can be ignored (perhaps changing the constant inside the exponent).
\end{proof}
\begin{thm}
\label{thm:poly}Let $k,l,m\in\mathbb{N}$, $m\ge2$. Let $R$ be
be given by taking at least $(3l)^{7m^{2}k^{4}}$ random reduced words
of length $l$ in the letters $\{x_{1},\dotsc,x_{m},x_{1}^{-1},\dotsc,x_{m}^{-1}\}$
independently, uniformly, with repetitions. Let $\Gamma=\langle x_{1},\dotsc,x_{m}|R\rangle$.
Then
\[
\lim_{l\to\infty}\mathbb{P}(\exists F,\rho:\Gamma\to\mathrm{GL}_{k}(F)\text{ such that }|\rho(\Gamma)|>2)=0
\]
where $F$ runs over the all fields, and where $\rho$ is a group
homomorphism.
\end{thm}

\begin{proof}
First apply Lemma \ref{lem:one F} with $F=\mathbb{C}$ and with $15m^{3}k^{4}\lceil\log l\rceil$
relators. We get that with high probability, any $\rho:\Gamma\to\mathrm{GL}_{k}(\mathbb{C})$
has $|\rho(\Gamma)|\le2$. Of course, the image of the generators
$\{x_{1},\dotsc,x_{m}\}$ is easy to characterise: we have $\rho(x_{i})^{2}=1$
$\forall i$ and for some $S\subset\{1,\dotsc,m\}$ we have $\rho(x_{i})=\rho(x_{j})$
$\forall i,j\in S$ and $\rho(x_{i})=1$ $\forall i\not\in S$. Compactly,
$|\{\rho(x_{i})\}\setminus\{1\}|\le1$.

Recall from the proof of Lemma \ref{lem:one F} the variety $X$ in
$\mathbb{C}^{2mk^{2}}$ and denote $p_{1},\dotsc p_{mk^{2}}$ the
polynomials defining it; and the notation $(A,A^{-1})$. The conditions
that matrices $A$ satisfy a given random word is a polynomial in
$2mk^{2}$ variables. Denote the polynomial that corresponds to the
$i^{\textrm{th}}$ word by $p_{mk^{2}+i}$ and let $M=mk^{2}+15m^{3}k^{4}\lceil\log l\rceil$.
Note that each of these polynomials has integer coefficients. We get
that 
\[
\bigcap_{i=1}^{M}Z(p_{i})=\{(A,A^{-1}):A_{i}^{2}=1\;\forall i,|A\setminus\{1\}|\le1\}.
\]
Denote the variety on the right by $Y$.

We now claim that $Y$ can by written as $\cap Z(r_{j})$ for some
$r_{1},\dotsc,r_{K}$ which depend only on $k$ and $m$, and in particular
do not depend on the field. Here is how: the condition $A_{i}^{2}=1$
corresponds to $k^{2}$ polynomials for each $i$. The condition 
\[
\bigcup_{S}\{A_{i}=A_{j}\forall i,j\in S,A_{i}=1\forall i\not\in S\}.
\]
gives us at most $(mk^{2})^{2^{m}}$ polynomials because for every
$S$ the corresponding variety is described by at most $mk^{2}$ polynomials,
but taking union requires to take every possible choice of a polynomial
for each $S$, and multiply them out. This describes our $r_{1},\dotsc,r_{K}$
(and gives $K\le mk^{2}+(mk^{2})^{2^{m}}$, but we will have no use
for this fact).

Now apply the effective nullstellensatz (Theorem \ref{thm:nullstellensatz})
$K$ times as follows. In all applications the polynomials $p_{i}$
from the nullstellensatz are our $p_{i}$, but the polynomial $r$
we take corresponding to the $r_{j}$ above. We get corresponding
$q_{i,j}$, $\nu_{j}$ and $b_{j}$, with the $b_{j}$ all satisfying
some bound, which we denote by $B$. Recall (\ref{eq:ENsimp}). The
number of variables is $2mk^{2}$ while the maximal value of the coefficients,
$h$, can be bounded roughly by $(2mk^{2})^{l}$. We get, 
\[
B\le\exp\left((2l)^{4m^{2}k^{4}+4mk^{2}+1}\cdot l\log(2mk^{2})\right)\le\exp\left((2l)^{7m^{2}k^{4}}\right)
\]
which holds for $l$ sufficiently large.

Consider now a field $F$ of characteristic larger than $B$. Then
\[
\sum p_{i}q_{i,j}=b_{j}r_{j}^{\nu_{j}}
\]
holds also in $F$, and because $\textrm{char}\,F>B\ge b_{j}$ we
get that $b_{j}\ne0$ in the field and we may divide by them. This
means that whenever $p_{i}=0$ for all $i$ so are $r_{j}$ for all
$j$, but that means that any $A$ which satisfy our first $15m^{3}k^{4}\lceil\log l\rceil$
words must also satisfy that $A_{i}^{2}=1$ $\forall i$ and that
$|A\setminus\{1\}|\le1$. So in $\textrm{GL}_{k}(F)$ too we get $|\langle A\rangle|\le2$.

Finally, for every prime $\tau$ smaller than $B$ apply Lemma \ref{lem:one F}
again, but this time with the field $F_{\tau}$ being the algebraic
closure of $\mathbb{Z}/\tau\mathbb{Z}$ and with $u=\lambda mk^{2}(2l)^{7m^{2}k^{4}}$
for some $\lambda$ to be fixed soon. We get that 
\[
\mathbb{P}(\exists\rho:\Gamma\to\mathrm{GL}_{k}(F_{\tau})\text{ s.t. }|\rho(\Gamma)|>2)\le\exp(-cu/mk^{2})=\exp(-c\lambda(2l)^{7m^{2}k^{4}}).
\]
Summing over all $\tau<B$ gives
\begin{multline*}
\mathbb{P}(\exists\tau<B\exists\rho:\Gamma\to\textrm{GL}_{k}(F_{\tau})\text{ such that }|\rho(\Gamma)|>2)\le B\exp(-c\lambda(2l)^{7m^{2}k^{4}})\\
\le\exp((1-c\lambda)(2l)^{7m^{2}k^{4}}).
\end{multline*}
Taking $\lambda=2/c$ we get that this probability goes to zero. Moving
from $F_{\tau}$ to a general field of characteristic $\tau$ is done
using the (usual, non-effective) nullstellensatz: find polynomials
$q_{i,j}\in\mathbb{Z}/\tau\mathbb{Z}[x_{1},\dotsc,x_{2mk^{2}}]$ such
that $\sum p_{i}q_{i,j}=r_{j}^{\nu_{j}}$ in $\mathbb{Z}/\tau\mathbb{Z}$
with the same $p_{i}$ and $r_{1},\dotsc,r_{K}$ as above, and note
that the existence of these $q_{i,j}$ ensures that in any field $F$
of characteristic $\tau$, if $A_{1},\dotsc,A_{m}$ are in $\textrm{GL}_{k}(F)$
and satisfy all words in $R$ then $|\langle A\rangle|\le2$, proving
the theorem.
\end{proof}

\subsection*{Acknowledgements}

We thank Nir Avni for the idea to use the effective nullstellensatz.
We thank Ron Livne for help with B\'ezout's theorem. We thank Tsachik
Gelander, Shahar Mozes and Michael Ben Or for many interesting discussions.

GK was supported by the Israel Science Foundation, by the Jesselson
Foundation and by Paul and Tina Gardner. AL was supported by the Israel
Science Foundation, by the National Science Foundation and by the
European Research Council.

\end{document}